\documentclass[11pt]{amsart}

\usepackage{amsmath,amssymb}
\usepackage[margin=3cm]{geometry}
\usepackage{booktabs}
\usepackage{url}

\newtheorem{theorem}{Theorem}
\newtheorem{lemma}[theorem]{Lemma}
\newtheorem{corollary}[theorem]{Corollary}
\newtheorem{proposition}[theorem]{Proposition}
\theoremstyle{remark}

\newcommand{\R}{\mathbb{R}}
\newcommand{\norm}[1]{\lVert #1\rVert}
\newcommand{\abs}[1]{\lvert #1\rvert}

\begin{document}

\title{Kusner's conjecture on equilateral sets is false}

\author{Logan R. Chalmers}
\address{University of Otago, Dunedin 9016, New Zealand}
\email{logan.chalmers@postgrad.otago.ac.nz}


\begin{abstract}
We disprove Kusner's 1983 conjecture that every equilateral set in
$\ell_p^n$ with $2<p<\infty$ has at most $n+1$ points: there exist $58$
points in $\R^{56}$ whose pairwise $\ell_5$ distances are all equal, so
the maximum equilateral-set size satisfies $e(\ell_5^{56})\ge58>57$.
This is the first equilateral set of more than $n+1$ points in $\ell_p^n$
for any finite $p\ge2$. The construction persists on an open interval of
exponents around $5$. The configuration is specified by the unique solution of an
explicit polynomial system with rational coefficients in a rational
box; existence and uniqueness are proved by exact integer inequalities.
The existence of the $58$-point equilateral set in $\ell_5^{56}$ has been formalised in Lean from standard axioms.
\end{abstract}

\maketitle

\section{Introduction}\label{sec:intro}

In correspondence in the early 1980s, Robert Kusner sent Richard Guy
questions about the largest equilateral sets in $\ell_p^n$. Guy
published them in 1983 \cite{Guy1983}, and the assertion
$e(\ell_p^n)=n+1$ for finite $p>2$ became known as Kusner's conjecture.
We disprove it at $p=5$. For $x\in\mathbb R^n$, write
\[
 \norm{x}_p=\left(\sum_{r=1}^n\abs{x_r}^p\right)^{1/p},
\]
and a set is \emph{equilateral} if the distance between every two distinct
points is the same. We write $e(X)$ for the largest size of an equilateral
set in a finite-dimensional normed space $X$.

For $1\le p<\infty$, the standard basis together with a suitable
multiple of $(1,\ldots,1)$ gives $n+1$ equilateral points. The Euclidean
maximum is $n+1$, while the $2^n$ cube vertices are equilateral
in $\ell_\infty^n$. Guy recorded a heuristic for the conjecture: each
new point must lie on the smooth $\ell_p$-spheres centred at the
preceding points, and in general position each additional sphere removes
one degree of freedom. Moreover, $e(X)$ bounds the Borsuk number of $X$
from below and equals the maximum number of pairwise touching
translates of the unit ball \cite{Kalai2015,LemmensParsons2015}.

Swanepoel proved the conjecture at $p=4$ and, for each $1<p<2$,
constructed equilateral sets of at least $(1+\varepsilon_p)n$ points for all
sufficiently large $n$ \cite{Swanepoel2004}. Smyth obtained an effective
interval around $2$ \cite{Smyth2013}, and Swanepoel made explicit an
$n$-dependent interval around $4$ \cite{Swanepoel2014}. Ge, Xu and Zhou
proved the full range $2\le p\le4$ \cite{GeXuZhou2026}. In the same range, Chalmers~\cite{PettyLp2026}
proved Petty's conjecture for real subspaces of $L_p$: every
$d$-dimensional real subspace contains $d+1$ equidistant points.
More strongly, every equilateral set in such a subspace extends to
an equilateral set of $d+1$ points.
At $p=5$, the
first integer exponent beyond this range, the best known upper bound is
$O(n\log n)$ \cite{AlonPudlak2003,GeXuZhou2026}. No equilateral set of
more than $n+1$ points in $\ell_p^n$ was known for finite $p\ge2$.

\begin{theorem}\label{thm:main}
There is an equilateral set of $58$ points in $\ell_5^{56}$. Consequently
\[
 e(\ell_5^{56})\ge58>57=56+1,
\]
so Kusner's conjecture is false.
\end{theorem}

{\emergencystretch=2em Theorem~\ref{thm:main} has been formalised in Lean from the standard axioms of
propositional extensionality (\texttt{propext}), classical choice
(\texttt{Classical.choice}), and quotient soundness (\texttt{Quot.sound}).\par}

\begin{corollary}\label{cor:interval}
There is an $\varepsilon>0$ such that
\[
 e(\ell_p^{56})\ge58
 \qquad\text{whenever}\qquad \abs{p-5}<\varepsilon.
\]
In particular, if
\[
 p_0=\inf\{p>2:e(\ell_p^n)>n+1\text{ for some }n\in\mathbb N\},
\]
then $p_0\in[4,5)$.
\end{corollary}

Searches over discrete alphabets and configurations with imposed symmetry
produced no candidate. We therefore minimized the following scale-invariant
equidistance residual over unrestricted coordinates. For $m=n+2$ points
$x^{(1)},\ldots,x^{(m)}$, put
\[
 d_{ij}=\sum_{r=1}^n\abs{x_r^{(i)}-x_r^{(j)}}^5.
\]
Writing $\bar d$ for their mean, set
\[
 R=\left(\binom{m}{2}^{-1}\sum_{i<j}
 \left(\frac{d_{ij}}{\bar d}-1\right)^2\right)^{1/2}.
\]
The search produced an approximate configuration of $102$ points in
dimension $100$. We repeatedly removed one point and one coordinate and
re-solved, obtaining the $58$-point configuration in dimension $56$; no
minimality in the dimension is claimed.

The numerical search produced an approximate centre and inverse
Jacobian. For the verification, $1596$ of the $3248$
coordinate entries are fixed at dyadic rational values. The remaining
$1652$ coordinates and the common fifth-power distance form a square
system of $1653$ equations. The coordinate gaps fix every sign throughout
a rational box, so the equations are rational polynomials there. Exact
integer bounds show that $z\mapsto z-BF(z)$ contracts the box. The fixed
point is an exact zero with positive distance coordinate, and hence gives
$58$ distinct equilateral points. From the certificate data,
the archived verification code reconstructs the integers in
Proposition~\ref{prop:certificate} and checks the displayed inequalities
\cite{Certificate}.

\section{The configuration}\label{sec:configuration}

Set $m=58$, $n=56$ and
\[
 P=\binom{m}{2}=1653.
\]
Order the pairs $(i,j)$, $1\le i<j\le m$, lexicographically. The
certificate is specified by four integer arrays in the archive:
\begin{itemize}
\item $K\in\mathbb Z^{58\times56}$, giving centre coordinates
      $c_{ir}=K_{ir}/2^{50}$;
\item \texttt{dK}, a one-entry array whose entry $d_K\in\mathbb Z$
      gives the centre value $D_c=d_K/2^{50}$ of the common fifth-power
      distance;
\item \texttt{sel}, an array of length $1653$: its first $1652$
      entries encode an ordered list $S$ of distinct coordinate
      positions $(i,r)$ as $56(i-1)+(r-1)$, and its final entry is the
      marker $3248$ for the variable $D$; and
\item \texttt{BI} $=B_0\in\mathbb Z^{P\times P}$, giving the rational matrix
      $B=B_0/2^{40}$.
\end{itemize}
The rows of $B_0$ follow the variable order specified by \texttt{sel}, and
its columns follow the lexicographic pair order.

Let $z\in\R^P$. Its first $1652$ entries are the coordinates $x_{ir}$
for $(i,r)\in S$, in the archived order, and its last entry is $D$.
Coordinates not in $S$ are fixed at $K_{ir}/2^{50}$. In this way $z$
determines points $x^{(1)},\ldots,x^{(58)}\in\R^{56}$ and a number $D$.
Define $F:\R^P\to\R^P$ by
\begin{equation}\label{eq:Fabs}
 F_{ij}(z)=\sum_{r=1}^{56}
 \abs{x_r^{(i)}-x_r^{(j)}}^5-D,
 \qquad 1\le i<j\le58,
\end{equation}
with components in lexicographic pair order. A zero of $F$ with $D>0$
gives $58$ distinct points at the common $\ell_5$ distance $D^{1/5}$.

Let $c\in\R^P$ be the vector of selected centre coordinates followed by
$D_c$, let
\[
 \rho=2^{-33},
 \qquad
 X=\{z\in\R^P:\norm{z-c}_\infty\le\rho\}.
\]
The minimum difference between two centre coordinates in the same
coordinate direction is
\begin{equation}\label{eq:gap}
 g=\min_{r}\min_{i<j}\abs{c_{ir}-c_{jr}}
   =\frac{215220270}{2^{50}}>2\rho.
\end{equation}
Each difference $x_r^{(i)}-x_r^{(j)}$ changes by at most $2\rho$ on
$X$, so its sign is fixed there. Write this sign as $s_{ijr}$. On $X$,
equation \eqref{eq:Fabs} is therefore the rational polynomial system
\begin{equation}\label{eq:Fpoly}
 F_{ij}(z)=\sum_{r=1}^{56}s_{ijr}
 \bigl(x_r^{(i)}-x_r^{(j)}\bigr)^5-D.
\end{equation}
In particular, $F$ is differentiable on a neighbourhood of $X$.

For a selected coordinate $x_{ir}$, the only nonzero entries in its
Jacobian column occur in pair rows containing $i$. In the row $(i,j)$,
\begin{equation}\label{eq:jacobian}
 \frac{\partial F_{ij}}{\partial x_{ir}}
 =5s_{ijr}(x_{ir}-x_{jr})^4,
 \qquad
 \frac{\partial F_{ij}}{\partial x_{jr}}
 =-5s_{ijr}(x_{ir}-x_{jr})^4,
\end{equation}
when the corresponding coordinate is selected, and
$\partial F_{ij}/\partial D=-1$.

\section{A contraction lemma}\label{sec:contraction}

For matrices, $\norm{\cdot}_\infty$ is the maximum absolute row sum.

\begin{lemma}\label{lem:contraction}
Let $X=\{z\in\R^P:\norm{z-c}_\infty\le\rho\}$, let $F$ be continuously
differentiable on a neighbourhood of $X$, with Jacobian $J$, and let
$B\in\R^{P\times P}$. Suppose
\[
 \norm{BF(c)}_\infty\le\eta,
 \qquad
 \sup_{z\in X}\norm{I-BJ(z)}_\infty\le q<1,
 \qquad
 \eta+q\rho<\rho.
\]
Then $F$ has exactly one zero $z_*$ in $X$. Moreover,
$\norm{z_*-c}_\infty\le\eta+q\rho$, and both $B$ and $J(z_*)$ are
invertible.
\end{lemma}

\begin{proof}
Set $T(z)=z-BF(z)$. For $z,w\in X$, the line segment between them lies
in $X$, and
\[
 T(z)-T(w)=\int_0^1
 \bigl[I-BJ(w+t(z-w))\bigr](z-w)\,dt.
\]
Hence $\norm{T(z)-T(w)}_\infty\le q\norm{z-w}_\infty$. Also
\[
 \norm{T(z)-c}_\infty
 \le q\norm{z-c}_\infty+\norm{BF(c)}_\infty
 \le q\rho+\eta<\rho,
\]
so $T$ maps the complete metric space $X$ into itself and is a
contraction. It has a unique fixed point $z_*$, and
$\norm{z_*-c}_\infty\le\eta+q\rho$.

For every $z\in X$, the inequality $\norm{I-BJ(z)}_\infty<1$ makes
$BJ(z)$ invertible by a Neumann series. Since the matrices are square,
$B$ and $J(z)$ are invertible. Thus $BF(z_*)=0$ implies $F(z_*)=0$.
Every zero of $F$ is a fixed point of $T$, proving uniqueness.
\end{proof}

\section{The certificate}\label{sec:certificate}

Put
\[
 Q=2^{50},\qquad R=2^{17},\qquad \rho=R/Q.
\]
For the lexicographically ordered pair $k=(i,j)$, define
\begin{equation}\label{eq:fhat}
 f_k=\sum_{r=1}^{56}\abs{K_{ir}-K_{jr}}^5-d_KQ^4.
\end{equation}
Then $F(c)=f/Q^5=f/2^{250}$.

Let $\widehat J\in\mathbb Z^{P\times P}$ be the numerator matrix for
$J(c)=\widehat J/Q^4$. Formula \eqref{eq:jacobian} determines its
coordinate columns: if the selected variable is $x_{ir}$ and the row is
$(i,j)$, its entry is
\[
 5\,\operatorname{sgn}(K_{ir}-K_{jr})
 \abs{K_{ir}-K_{jr}}^4,
\]
with the opposite sign for the $x_{jr}$ column. All other coordinate
entries are zero, and the $D$ column consists of $-Q^4$.

For a pair row $k=(i,j)$ and a coordinate direction $r$, let
$w\in\{0,1,2\}$ be
the number of selected positions among $(i,r)$ and $(j,r)$. Each
Jacobian entry arising from a selected endpoint has numerator radius
\begin{equation}\label{eq:Ehat}
 5\left[(\abs{K_{ir}-K_{jr}}+wR)^4
       -\abs{K_{ir}-K_{jr}}^4\right].
\end{equation}
Placing these radii in the corresponding positions defines a
nonnegative integer matrix $\widehat E$ such that, entrywise,
\begin{equation}\label{eq:Jbound}
 \abs{J(z)-J(c)}\le \widehat E/Q^4
 \qquad (z\in X).
\end{equation}
Indeed, \eqref{eq:gap} keeps the interval away from zero, and
$t\mapsto t^4$ is increasing and convex for $t\ge0$, so the largest
deviation occurs at the upper endpoint used in \eqref{eq:Ehat}.

Define
\begin{align}
 f_0&=\max_k\abs{f_k},\label{eq:f0}\\
 \eta_0&=\max_a\abs{(B_0f)_a},\label{eq:eta0}\\
 q_0&=\max_a\sum_b\left(
 \abs{2^{240}\delta_{ab}-(B_0\widehat J)_{ab}}
 +(\abs{B_0}\widehat E)_{ab}\right).\label{eq:q0}
\end{align}
Here $\abs{B_0}$ denotes entrywise absolute value. These definitions give
\begin{equation}\label{eq:scaled}
 \norm{F(c)}_\infty=\frac{f_0}{2^{250}},
 \qquad
 \eta:=\norm{BF(c)}_\infty=\frac{\eta_0}{2^{290}},
 \qquad
 q:=\frac{q_0}{2^{240}}.
\end{equation}
Moreover, $\sup_{z\in X}\norm{I-BJ(z)}_\infty\le q$. This follows from
\eqref{eq:Jbound} and
$\abs{B(J(z)-J(c))}\le\abs{B}\widehat E/Q^4$ entrywise.

\begin{proposition}\label{prop:certificate}
The archived integer arrays satisfy the following statements.
\begin{enumerate}
\item The array shapes are
$K:(58,56)$, $\mathrm{dK}:(1,)$, $\mathrm{sel}:(1653)$ and
$B_0:(1653,1653)$; the last entry of $\mathrm{sel}$ is $3248$,
its preceding entries encode
distinct valid coordinate positions, and the pair list contains every
$1\le i<j\le58$ exactly once.
\item The minimum gap is the value in \eqref{eq:gap}, and hence every
coordinate difference has fixed nonzero sign on $X$.
\item Every $z\in X$ satisfies $D>1-2^{-32}>0$.
\item The integers in \eqref{eq:f0}--\eqref{eq:q0} satisfy
\begin{equation}\label{eq:integerchecks}
 f_0<2^{203},\qquad
 \eta_0<2^{247},\qquad
 q_0<2^{226},\qquad
 \eta_0+2^{17}q_0<2^{247}.
\end{equation}
\end{enumerate}
Consequently
\begin{equation}\label{eq:headlines}
 \norm{F(c)}_\infty<2^{-47},\qquad
 \eta<2^{-43},\qquad
 q<2^{-14},\qquad
 \eta+q\rho<2^{-10}\rho<\rho.
\end{equation}
\end{proposition}

\begin{proof}
All assertions reduce to integer comparisons.

For (1), direct inspection verifies the four array shapes and integer types, the
final marker, and the range and uniqueness of the preceding entries of
$\mathrm{sel}$. The pair list is generated by the nested order
$i=1,\ldots,58$ and $j=i+1,\ldots,58$.

For (2), $\min_r\min_{i<j}\abs{K_{ir}-K_{jr}}=215220270$.
Since a difference moves by at most $2R$ over the common denominator
$Q$, comparison with $2R$ proves sign constancy. For the third,
$d_K-R>Q-2^{18}$ is exactly the displayed lower bound on $D$.

Finally, $f$ is formed by \eqref{eq:fhat}, $\widehat J$ by
\eqref{eq:jacobian}, and $\widehat E$ by \eqref{eq:Ehat}. Integer matrix
multiplication then forms \eqref{eq:eta0} and \eqref{eq:q0}. Direct
comparison gives the four inequalities in \eqref{eq:integerchecks}.
After division by the denominators in \eqref{eq:scaled}, the first three
give the first three bounds in \eqref{eq:headlines}. Since
$\rho=2^{-33}$, the last comparison is precisely
$\eta+q\rho<2^{-43}=2^{-10}\rho$.
\end{proof}

\begin{proof}[Proof of Theorem~\ref{thm:main}]
By Proposition~\ref{prop:certificate}, $F$ is the polynomial system
\eqref{eq:Fpoly} throughout $X$, and the hypotheses of
Lemma~\ref{lem:contraction} hold. Hence $F$ has a unique zero $z_*$ in
$X$. Its distance coordinate $D_*$ is positive. Therefore
\[
 \sum_{r=1}^{56}\abs{x_r^{(i)}-x_r^{(j)}}^5=D_*>0
 \qquad (1\le i<j\le58).
\]
The corresponding $58$ points are distinct and all their $\ell_5$
distances equal $D_*^{1/5}$. Thus $e(\ell_5^{56})\ge58$.
\end{proof}

The configuration is determined by $z_*$, not by the rounded centre
$c$. Since $J(z_*)$ is invertible, $z_*$ is an isolated complex zero of
the rational polynomial system. By elimination, its coordinates are real
algebraic; the system and box represent them exactly.

\section{Consequences}\label{sec:consequences}

\begin{corollary}\label{cor:family}
The configuration in Theorem~\ref{thm:main} belongs to a real-analytic
local family of dimension $1596$ of equilateral $58$-point
configurations in $\ell_5^{56}$.
\end{corollary}

\begin{proof}
Let $\mathcal F:\R^{3249}\to\R^{1653}$ be the equidistance map in which
all $3248$ coordinate entries and the common fifth-power distance are
variables. Near the certified configuration every coordinate difference
has fixed sign, so $\mathcal F$ is real analytic. Its derivative with
respect to the $1652$ selected coordinates and the distance variable is
$J(z_*)$, which is invertible by Lemma~\ref{lem:contraction}. The implicit
function theorem therefore expresses these $1653$ variables locally as
real-analytic functions of the remaining $1596$ coordinates. Positivity
of the common distance persists after restricting the neighbourhood.
\end{proof}

The orbit under translations and dilations has dimension at most $57$;
hence the family contains configurations not related to the certified
one by translation and dilation.

\begin{proof}[Proof of Corollary~\ref{cor:interval}]
For $p$ near $5$, define
\[
 F_{ij}(p,z)=\sum_{r=1}^{56}
 \exp\!\left(p\log\bigl[s_{ijr}(x_r^{(i)}-x_r^{(j)})\bigr]\right)-D.
\]
The quantity inside each logarithm is positive on $X$ by
\eqref{eq:gap}. Thus $F(p,z)$ is real analytic near $(5,z_*)$, and
$F(5,z_*)=0$. The proof of Theorem~\ref{thm:main} gives
$\norm{I-BD_zF(5,z_*)}_\infty<1$, so $D_zF(5,z_*)$ is invertible. The
implicit function theorem supplies an analytic solution $z(p)$ for all
$p$ in an open interval containing $5$. By continuity, after shrinking
the interval if necessary, $z(p)$ remains inside $X$ and its
$D$-coordinate remains positive. It therefore defines $58$ distinct
equilateral points in $\ell_p^{56}$.

It follows that $p_0<5$. The theorem of Ge, Xu and Zhou
\cite{GeXuZhou2026} gives $p_0\ge4$, and hence $p_0\in[4,5)$.
\end{proof}

At $p=1$, the signed standard basis gives $e(\ell_1^n)\ge2n$, equality
is known for $n\le4$, and $e(\ell_1^n)\le Cn\log n$
\cite{AlonPudlak2003,GeXuZhou2026}. Whether equality holds for all $n$
remains open: either $e(\ell_1^n)>2n$ in some dimension, or the upper
bound can be sharpened from $Cn\log n$ to $2n$.

\section*{Competing interests and funding}

The author reports no competing interests. This research did not receive
any specific grant from funding agencies in the public, commercial or
not-for-profit sectors.

\section*{Data and code availability}

The certificate data, exact integer-arithmetic verifier, explanatory
notebooks, and instructions for reproducing the verification are openly
available at \url{https://doi.org/10.5281/zenodo.21911503}.
The Lean formalisation of Theorem~\ref{thm:main}, including the complete source
and instructions for checking the proof, is available at
\url{https://doi.org/10.5281/zenodo.23786173}.

\section*{Disclosure of AI use}

OpenAI's GPT-5.6 Sol assisted in implementing the computational search
strategy and drafting this manuscript. OpenAI's Astra assisted with implementing
the Lean formalisation. The author takes full
responsibility for the content of the article.

\end{document}